\documentclass{amsart}

\usepackage{textcomp}
\usepackage{graphicx} 
\usepackage{latexsym} 
\usepackage{hyperref} 
\usepackage{amssymb, mathrsfs, amsfonts, amsmath} 
\usepackage{amsbsy} 
\usepackage{amsfonts} 

\usepackage{color}

\newtheorem{corollary}{Corollary} 
\newtheorem*{theorem*}{Theorem} 
 
\newtheorem{lemma}{Lemma} 
\newtheorem{proposition}{Proposition} 
\newtheorem{remark}{Remark}
 
\newtheorem{theorem}{Theorem} 
 
\newtheorem{conj}{Conjecture}

\numberwithin{equation}{section}

\newcommand{\R}{\mathbb R}

\newcommand{\s}{\mathbb S}

\newcommand{\interno}[2]{\left\langle #1 ,#2 \right\rangle}
\newcommand{\divergence}{\text{div}}

\begin{document}

\title[Uniqueness of free-boundary minimal hypersurfaces]{Uniqueness of free-boundary minimal hypersurfaces in rotational domains}


\author[Barbosa]{Ezequiel Barbosa}
\address{Instituto de Ci\^{e}ncias Exatas-Universidade Federal de Minas Gerais\\ 30161-970-Belo Horizonte-MG-BR} \email{ezequiel@mat.ufmg.br} 

\author[Freitas]{Allan Freitas} \address{Departamento de Matem\'{a}tica-Universidade Federal da Para\'{i}ba\\ 58059-900  Jo\~{a}o Pessoa, Para\'{i}ba, Brazil}
\email{allan@mat.ufpb.br}

\author[Melo]{Rodrigo Melo} \address{Universidade Federal de Alagoas,\\ Campus do Sert\~{a}o, Rodovia AL 145, Prefeito Jos\'e Serpa de Menezes S/N, Cidade Universit\'aria, 57.480-000.
Delmiro Gouveia - AL - Brazil.}
\email{rodrigo@pos.mat.ufal.br}

\author[Vit\'orio]{Feliciano Vit\'{o}rio} \address{Instituto de Matem\'{a}tica, Universidade Federal de Alagoas,\\ Campus A. C. Sim\~{o}es, BR 104 - Norte, Km 97, 57072-970.
Macei\'o - AL -Brazil.}
\email{feliciano@pos.mat.ufal.br}



\date{}

\dedicatory{}

\begin{abstract}
In this work, we investigate the existence of compact free-boundary minimal hypersurfaces immersed in several domains. Using an original integral identity for compact free-boundary minimal hypersurfaces that are immersed in a domain whose boundary is a regular level set, we study the case where this domain is a quadric  or, more generally, a rotational domain. This existence study is done without topological restrictions. We also obtain a new gap theorem for free boundary hypersurfaces immersed in an Euclidean ball and in a rotational ellipsoid.
\end{abstract}

\maketitle

Minimal surfaces have been one of the most studied objects in differential geometry. The problem of showing the existence of a area-minimizing surface whose boundary is a given  Jordan curve was introduced by Plateau in 18 century and it was completely solved in 1930 decade, independently by Douglas \cite{douglas1931} and Rad\'o \cite{rado1930}. These two works inspired Courant \cite{courant1940} in 1940 to investigate the existence of area-minimizing surfaces immersed in a domain $U\subset\R^3$ whose boundaries are lying on $\partial U$. This type of problem have been known as {\it free boundary problem}.

In the  decades after Courant's work, several contributions have appeared: Nitsche \cite{nitsche1985} was the first to study minimal surfaces with free boundary in the unitary ball $\mathbb{B}^3\subset\R^3$;  Gr\"uter and Jost \cite{gruter1986} and Struwe \cite{struwe1984} obtained several existence results for free boundary minimal disks on convex domains; a general existence theory of free boundary minimal disks was later developed by Fraser \cite{fraser2000}.

In the particular case where the domain is the unitary ball $\mathbb{B}^3$, the simplest examples of minimal free boundary surfaces are the equatorial disk $\mathbb{D}$ and the critical catenoid $\mathbb{K}$.  

After the work of Nitsche in 1985 there was no development in the theory for the specific case of free boundary  minimal surfaces on  $\mathbb{B}^3\subset\R^3$. Only in 2010's decade, with the works of Fraser and Schoen (\cite{fraser2011}, \cite{fraser2016}) that the topic started to be explored again. They have found a connection between free boundary minimal surfaces on the unitary ball $\mathbb{B}^n\subset\R^n$ and Steklov eigenvalues.




In his work at 1985, Nitsche  claimed without proof that the critical catenoid is the only free boundary minimal annulus in $\mathbb{B}^3$.  The conjecture was restated again in 2014 by Fraser and Li:

\begin{conj}(Fraser-Li \cite{fraser2014}) \label{conj}
Up to congruence, the critical catenoid is the only properly embedded free boundary minimal annulus in $\mathbb{B}^3$.
\end{conj}

Besides the conjecture still being unanswered up today, several important characterization results have come in the past years. In 2016 Fraser and Shoen \cite{fraser2016} have proved the conjecture under the hypothesis that the first Steklov eigenvalue $\sigma_1$ is equal to 1. In 2018, McGrath \cite{mcgrath2018} showed that $\sigma_1=1$ if the annulus is symmetric with respect to coordinate planes. Recently, Kusner and Mcgrath \cite{kusner2020} have weakened the symmetric hypothesis asking that the annulus has been symmetric with respect to the origin. Another characterization, without topological restrictions, uses Bj\"orling unicity Theorem to assure that if one boundary component $\gamma$ of a minimal surface with free boundary in $\mathbb{B}^3$ is rotationally invariant, then the surface is either the equatorial disk $\mathbb{D}$ or the critical catenoid $\mathbb{K}$ (see corollary 3.9 of \cite{kapouleas2017} by Kapouleas and Li).

In addition to this lot of interest in studying free-boundary submanifolds in the unit ball, and aiming to expand this theory in new directions, many works have been developed to approach the case  where the free-boundary minimal hypersurface is embedded in other domains. For example, the free-boundary problem already was studied where the boundary of domain is a wedge (\cite{lopez}), a slab (\cite{souam}), a convex cone (\cite{choe}), a cylinder (\cite{lopez-pyo}) and many others. More generally, Wheeler \cite{wh} by studying the mean curvature flow of embedded disks with free
boundary in embedded cylinder or generalized cone of revolution, proved the existence of minimal free boundary disks in such domains without requiring topological obstructions. In this scope, we also cite the relevant work \cite{maximo}  by M\'aximo, Nunes and Smith, where the authors proved the existence of free boundary minimal annuli inside convex subsets in three-dimensional Riemannian manifolds of nonnegative Ricci curvature.

In this work, we deal with domains $\Omega\subset\mathbb{R}^{n}$ whose boundary is a regular level set of a function $F$, i.e., $\partial\Omega=F^{-1}(1)$. For such domains we obtain a relevant Minkowski-type identity for compact free-boundary minimal hypersurfaces contained in it (see Proposition \ref{intid}). We use this identity to study the particular case where the function $F$ is a quadratic polynomial,
\begin{equation}
\label{quadraticpoly}
F(x_1, x_2, ..., x_n)=\sum_{i=1}^na_ix_i^2+bx_n+c,
\end{equation}
with $a_i \in \left\{-1,0,1\right\}$ and $b, c\in \R$, and therefore, the boundary of domain is a quadric hypersurface. We remark that this approach is used to study the existence of hypersurfaces that are inside or outside from $\Omega=F^{-1}(1)$. This permits unify the study of some remarkable domains such as cones, circular paraboloids, parabolic cylinders, slabs, hyperboloid of one sheet and many others. In particular, we obtain the following result:

\vspace{0.2cm}
\textbf{Theorem A}: 
Let $F$ be as in (\ref{quadraticpoly}) with $b=0$, $c\leq0$ and at least one of the coefficients (say $a_n$) being different from 1. There exists a minimal hypersurface $\Sigma$ with free boundary on $\partial \Omega$ if and only if $a_1=a_2=...=a_{n-1}=1$. Besides, 
\begin{enumerate}
	\item[(a)] $a_n=-1 \ \Rightarrow \ \Sigma$ is the flat disk supported at the origin.
	\item[(b)] $a_n=0 \ \Rightarrow \ \Sigma$ are the flat disks intersecting $\partial\Omega$ orthogonally.
\end{enumerate}   
\vspace{0.2cm}

Analysing the signal of the above coefficients $a_{i},b,c$, we obtain others nonexistence results for quadrics domains (see Theorems \ref{bneq0}, \ref{2sheets} and \ref{cilindcone}) and uniqueness results for rotational graphics domains (see Theorem \ref{rotationalgraphic}). 

A large interest in rigidity results for free-boundary submanifolds in several domains and with suitable hypotheses also has directed research interest. In this sense, many known results in closed submanifolds has inspired an approach in this scope. For example, the Conjecture \ref{conj} can be viewed as a free-boundary's version from the Lawson's conjecture about closed minimal surfaces in the sphere. In the direction to affirm this result, we obtain the following characterization, without topological assumptions:

\vspace{0.2cm}
\textbf{Theorem B}: Let $\Sigma\subset\mathbb{B}^{3}$ be a embedded minimal free boundary surface and let $\gamma\subset\partial\Sigma$ be a connected component of the boundary of $\Sigma$. Assume that exists one collar $\Gamma\subset\Sigma$ over $\gamma$ whose the set  $\left\{x^T, Ax^T\right\}$ is linearly dependent, where $A$ is the shape operator. Then, $\Sigma$ is either the equatorial disk $\mathbb{D}$ or the critical catenoid $\mathbb{K}$.

\vspace{0.2cm}

Also, the classical gap results for minimal submanifolds immersed in spheres by  Chern-do Carmo-Kobayashi have been evoked in order to establish similar results in the free-boundary's context, where, now, the ambient space is an Euclidean ball (see \cite{ambrozio2016}, \cite{celso}, \cite{feliciano}  and references therein). In this work, we also intend to approach gap results where the ambient space is a rotational ellipsoid or a ball. More specifically, we obtain the following results

\vspace{0.2cm}
\textbf{Theorem C}: Let $\Sigma^2$ be a compact minimal surface with free boundary in the ellipsoid 
$$\frac{x_1^2}{a^2}+\frac{x_2^2}{a^2}+\frac{x_3^2}{b^2}\leq 1$$
with $a\geq b$. If 
$$|A|^2\left[\interno{x}{N}+\left(\frac{a^2}{b^2}-1\right)\interno{x}{E_3}\interno{N}{E_3}\right]^2\leq2$$
on $\Sigma$, then $\Sigma$ is homeomorphic to a disk or an annulus.
\vspace{0.2cm}

\vspace{0.2cm}
\textbf{Theorem D}: Let $\Sigma^n$ be a free boundary minimal hypersurface in  the  unit ball $\mathbb{B}^{n+1}$. If $|A|^2\leq 2n$,  then $M$ is a totally geodesic equatorial disk $\mathbb{D}^n$.
\vspace{0.2cm}

This paper is organized as follows. In the first section, we approach some preliminaries for the underlying theme, obtaining a relevant integral identity for free boundary minimal hypersurfaces immersed in domains whose boundary is a regular value of a smooth function $F$. This permits obtain the results of the section 2, where we approach the particular case when $F$ is a quadratic polynomial and, thus, we have a quadric domain. In section 3, we continue investigating the applicability of the integral formula obtained in section 1, by studying rotational domains. In order to address the Conjecture \ref{conj}, we devote the section 4 to obtain a rigidity result for the critical catenoid immersed a 3-ball, with an additional condition. By the end, the Sectios 5 purposes to study some gap theorems (in the second fundamental form) for free-boundary minimal hypersufaces. In a first view, we extend a known gap result from Ambrozio-Nunes in a topological way, when the ambient is a rotational ellipsoid. And finally, we obtain a new gap result for free-boundary minimal hypersurfaces immersed in an Euclidean ball.

\section{Preliminaries}

Let $\Omega\subset \mathbb{R}^{n}$ be a domain with smooth boundary $\partial\Omega$ and denote by $\bar{N}$ the outward unit normal to $\partial\Omega$. We consider $\Sigma^{n-1}\hookrightarrow\Omega$ an hypersurface with boundary such that $\partial\Sigma\subset\partial\Omega$. We denote $N$ the outward unit normal to $\Sigma$ and $\nu$ the outward conormal along $\partial\Sigma$ in $\Sigma$. We remember that, in this scope, a hypersurface $\Sigma$ is called $\textit{Free-Boundary}$ if $\Sigma$
meets $\partial\Omega$ orthogonally. In others words, $\nu=\bar{N}$ along $\partial\Sigma$ or, equivalently, $\langle \bar{N}, N\rangle=0$ along $\partial\Sigma$. Since the classical Divergence Theorem, we can get the following integral identity to minimal free-boundary surfaces: 
\begin{equation}
\int_{\partial\Sigma}\langle x,\bar{N}\rangle\,ds=(n-1)|\Sigma|.\label{fbident}
\end{equation}
From now on, in the sequel, we obtain a new integral identity in the case where the domain has its boundary as a regular level set of a function $F:\mathbb{R}^{n}\longrightarrow\mathbb{R}$. More precisely, we have the following:

\begin{proposition}\label{intid} 
Let $\Omega$ be a domain in $\mathbb{R}^{n}$ such that $\partial\Omega=F^{-1}(1)$, 
where $F:\mathbb{R}^{n}\longrightarrow\mathbb{R}$ is a smooth map and $1$ 
is a regular value of $F$. Suppose that the normal vector field $\bar{N}$ of $\partial\Omega$ satisfies the condition
$$\bar{N}=\frac{\nabla F}{|\nabla F|}.$$
If $\Sigma$ is a free-boundary hypersurface contained in $\Omega$, with $\partial \Sigma\subset\partial\Omega$, and $\varphi\in C^{\infty}(\Sigma)$, 
then
\begin{equation}\label{fund}
\int_{\partial \Sigma}|\nabla F|\varphi\,ds=\int_{\Sigma}\varphi\Delta_{\Sigma}F\,d\Sigma+\int_{\Sigma}(1-F)\Delta_{\Sigma}\varphi\,d\Sigma,
\end{equation}
where $\Delta_{\Sigma}$ is the Laplace-Beltrami operator in $\Sigma$. 
In particular, with the above hypothesis,
$$\int_{\partial \Sigma}|\nabla F|\,ds=\int_{\Sigma}\Delta_{\Sigma}F\,d\Sigma.$$
\end{proposition}
\begin{proof}
We note that, once $\Sigma$ is free boundary and $\partial\Omega$ is a level set of $F$, $\nabla F=|\nabla F|\nu$ over $\partial\Sigma$. Then, it follows from the divergence Theorem that 
{\small $$\int_{\partial\Sigma}\varphi|\nabla F|\,ds=\int_{\partial\Sigma}\varphi\interno{\nabla F}{\nu}\,ds=\int_\Sigma\divergence(\varphi \nabla F)\,d\Sigma=\int_\Sigma\varphi\Delta_{\Sigma} F\,d\Sigma+\int_\Sigma\interno{\nabla\varphi}{\nabla F}\,d\Sigma.$$} 
On other hand, the divergence Theorem also gives 
{\small $$\int_\Sigma\Delta_\Sigma\varphi\,d\Sigma=\int_{\partial\Sigma}F\interno{\nabla \varphi}{\nu}\,ds=\int_\Sigma\divergence(F\nabla \varphi)\,d\Sigma=\int_\Sigma F\Delta_{\Sigma}\varphi\,d\Sigma +\int_\Sigma\interno{\nabla\varphi}{\nabla F}\,d\Sigma.$$}
The equation (\ref{fund}) now follows taking the difference of the two identities above.
\end{proof}


Talking in a more general way, we observe that the signal analysis of partial derivatives on the underlying function $F$ permits obtain a first result of nonexistence and uniqueness for free-boundary minimal surfaces immersed in it.  

\begin{proposition} 
Under the same conditions of the previous proposition, we have
\begin{itemize}
\item[(i)]Assume there exists $v\in \mathbb{R}^{n}$ such that either $\frac{\partial F}{\partial v}>0$ or $\frac{\partial F}{\partial v}<0$. Then there is no immersed minimal hypersurface with free boundary in $\Omega$.
\item[(ii)]
Assume there exists $v$ such that either $\frac{\partial F}{\partial v}\leq0$ or $\frac{\partial F}{\partial v}\geq0$. Then, $\frac{\partial F}{\partial v}=0$ on $\partial \Sigma$ and consequently $\Sigma$ is totally geodesic: $\Sigma$ is in an intersection $\Omega \cap \Pi$, where $\Pi$ is a plane.
\end{itemize}
\end{proposition}

\begin{proof}
Consider the function $\varphi=\interno{v}{x}$, where $v\in \mathbb{R}^{n}$. Then,
\[
\nabla_{\Sigma}\varphi=v-\interno{v}{N}N.
\]
Thus, on the boundary $\partial \Sigma$, we have that
\[
\frac{\partial \varphi}{\partial \nu}=\interno{v-\interno{v}{N}N}{\nu}=\interno{v}{\nu}=|\nabla F|^{-1}\interno{v}{\nabla F}.
\]
Since $\Sigma$ is free boundary: $\nu=|\nabla F|^{-1}\nabla F$ and $\interno{\nu}{N}=0$ on $\partial \Sigma$. If $\Sigma$ is minimal we have $\Delta_{\Sigma}\varphi=0$. Using integration by parts,
\[
0=\int_{\Sigma}\Delta_{\Sigma}\varphi\,d\Sigma=\int_{\partial\Sigma}\frac{\partial \varphi}{\partial \nu}\,ds=\int_{\partial\Sigma}|\nabla F|^{-1}\interno{v}{\nabla F}\,ds.
\]
Therefore, if either $\frac{\partial F}{\partial v}>0$ or $\frac{\partial F}{\partial v}<0$ we get a contradiction with the above identity. If either $\frac{\partial F}{\partial v}\leq0$ or $\frac{\partial F}{\partial v}\geq0$, we have
\[
\frac{\partial \varphi}{\partial \nu}=\frac{\partial F}{\partial v}=0\,, \mbox{on}\,\, \partial \Sigma\,.
\]
Then, since $\Delta_{\Sigma}\varphi =0$, we use the maximum principle (or integration by parts) to obtain that $\varphi=\interno{v}{x}$ is a constant  in $\Sigma$. This means that $\Sigma$ is contained in a plane.
\end{proof}

A slight change in the previous demonstration (doing $v=e_{i}$) shows that
\begin{corollary}\label{dF>0}
Under the same conditions of Proposition 1,
\begin{itemize}
\item[(i)] If exists $i$ such that either $\frac{\partial F}{\partial x_i}>0$ or $\frac{\partial F}{\partial x_i}<0$, then there is no immersed minimal hypersurface with free boundary in $\Omega$.
\item[(ii)] 
If exists $i$ such that either $\frac{\partial F}{\partial x_i}\leq0$ or $\frac{\partial F}{\partial x_i}\geq0$, then, $\frac{\partial F}{\partial x_i}=0$ on $\partial \Sigma$ and consequently $\Sigma$ is totally geodesic: $\Sigma$ is in an intersection $\Omega \cap \Pi$, where $\Pi$ is a plane.

\end{itemize}
\end{corollary}

\begin{remark}
Note that all the results above are true for free boundary minimal submanifolds $\Sigma^k$ in $\Omega$. Uniqueness and non-existence results above are true for the case of any codimensional submanifolds.
\end{remark}

If $F$ is $k$-homogeneous, i.e., $F(tx)=t^kF(x), \ \ \forall t\in\R$, then it is easy to check that $kF(x)=\interno{\nabla F}{x}$. In particular if  $x\in\partial\Omega$,
$$k\cdot 1=\interno{\nabla F}{x}=|\nabla F|\interno{\bar{N}}{x}.$$ 
Multiplying both sides by $|\nabla F|^{-1}$ and taking the integral over $\partial\Sigma$  we can use (\ref{fbident}) to get the following result.


\begin{proposition}\label{Fhomo}
Suppose that $F$ is $k$-homogeneous differentiable. If $\Sigma$ is minimal free-boundary then
\[
\int_{\partial \Sigma}k|\nabla F|^{-1}ds=(n-1)|\Sigma|.
\] 
\end{proposition}

We note that Theorem \ref{Fhomo} remains true even when a constant $c$ is added to $F$. Indeed, the maps $G=F+c$ and $F$ have the same gradient field over $\Sigma$.  


\section{Existence of Minimal Free-Boundary Surfaces in quadrics domains}

Now, we consider the important examples of domains whose boundary is a quadric surface. Quadric surfaces are frequently used in geometric modelling. In all situations, we consider domains $\Omega\subset\mathbb{R}^{n}$ such that $\partial\Omega=F^{-1}(1)$, where $F:\mathbb{R}^{n}\longrightarrow\mathbb{R}$ is the map 

\begin{equation}\label{Ffull}
F(x_1, x_2, ..., x_n)=\sum_{i=1}^na_ix_i^2+bx_n+c,
\end{equation}
with $a_i \in \left\{-1,0,1\right\}$ and $b, c\in \R$.

From now on, we are interested to study the existence of minimal free-boundary surfaces $\Sigma\subset\Omega$ in such domains by analysing some choices for the coefficients $a_{i}$, $b$ and $c$ and, in this sense, approaching some known quadric surfaces.  

When $a_n=0$ we have $\frac{\partial F}{\partial x_n}=b$. In this case, if $b\neq0$, it follows from Corollary\,\ref{dF>0} that there is no minimal free-boundary surfaces $\Sigma\subset\Omega$. Then, our first existence result is the following.


\begin{theorem} \label{bneq0}
If the map $F$ is such that $a_n=0\neq b$, then there does not exist any immersed minimal hypersurface with free boundary on $\partial \Omega$. 
\end{theorem}

The above theorem ensures, for example, that there does not exist any immersed minimal hypersurface with free boundary on circular paraboloid, hyperbolic paraboloid and parabolic cylinder.

In order to study more  cases which Theorem \ref{bneq0} has not covered, we need of some ingredients associated to $F$. First, we calculate the gradient field $\nabla F\in T\R^n$. Here, $\left\{E_1, E_2,..., E_n\right\}$ is the canonical basis of $\R^n$.
$$\nabla F=\sum_{i=1}^{n-1}2a_ix_iE_i+(2a_nx_n+b)E_n$$
and
$$|\nabla F|=\sqrt{\sum_{i=1}^{n-1}4a_i^2x_i^2+(2a_nx_n+b)^2}.$$
Once $\Sigma$ is free boundary, we note that the field $\nabla F$ is tangent to $\Sigma$ over $\partial\Sigma$. 

Now we calculate the Laplacian of $F$. Here, $N_1, N_2, ..., N_n$ are the coordinates of $N$ in $\R^n$.
$$\Delta_\Sigma F=\Delta_{\R^n} F-\interno{\overline{\nabla}_N\nabla F}{N}=\sum_{i=1}^{n}\left[2a_i-2a_iN_i^2\right]=2\sum_{i=1}^{n}a_i(1-N_i^2).$$ 

A straightfull computation also gives          
$$\interno{x}{\frac{\nabla F}{|\nabla F|}}=\frac{\sum_{i=1}^n2a_ix_i^2+bx_n}{|\nabla F|}=\frac{2(F(x)-c)-bx_n}{|\nabla F|}=\frac{2(1-c)-bx_n}{|\nabla F|}.$$

Therefore, Proposition \ref{intid} gives 
\begin{equation}\label{eq1}
\int_{\partial\Sigma}|\nabla F|\,ds=\int_{\Sigma}\Delta_\Sigma F\,d\Sigma=\int_\Sigma2\sum_{i=1}^{n}a_i(1-N_i^2)\,d\Sigma
\end{equation}
and equation (\ref{fbident}) gives

\begin{equation}\label{eq2}
\int_{\partial\Sigma}\frac{2(1-c)-bx_n}{|\nabla F|}\,ds=\int_{\partial\Sigma}\interno{x}{\frac{\nabla F}{|\nabla F|}}\,ds=\int_{\partial\Sigma}\interno{x}{\bar{N}}\,ds=(n-1)|\Sigma|.
\end{equation}
If we make the difference between (\ref{eq1}) and (\ref{eq2}) it follows that
\begin{eqnarray}
& &\int_{\partial\Sigma}\frac{4\displaystyle\sum_{i=1}^{n-1}a_i^2x_i^2+(2a_nx_n+b)^2+bx_n-2(1-c)}{|\nabla F|}\,ds \label{geral}\\ 
& &=\int_\Sigma\left[2\sum_{i=1}^na_i(1-N_i^2)-(n-1)\right]\,d\Sigma.\nonumber
\end{eqnarray}
\normalsize


\begin{theorem}[Theorem A] \label{all=1}
Let $F$ be as in (\ref{Ffull}) with $b=0$, $c\leq0$ and at least one of the coefficients (say $a_n$) being different from 1. There exists a minimal hypersurface $\Sigma$ with free boundary on $\partial \Omega$ if and only if $a_1=a_2=...=a_{n-1}=1$. Besides, 
\begin{enumerate}
	\item[(a)] $a_n=-1 \ \Rightarrow \ \Sigma$ is the flat disk supported at the origin.
	\item[(b)] $a_n=0 \ \Rightarrow \ \Sigma$ are the flat disks intersecting $\partial\Omega$ orthogonally.
\end{enumerate}   
\end{theorem}

\begin{proof}
Without loss of generality we can assume that the first $k$ coefficients $a_1,..., a_k$, are equal to 1. Then, $F$ can be written as 
\begin{equation}\label{b=0}
F(x_1,x_2,..., x_n)=\sum_{i=1}^kx_i^2+\sum_{i=k+1}^na_ix_i^2+c.
\end{equation}
Putting $b=0$ on the left side of (\ref{geral}) gives
\begin{eqnarray}
& &\displaystyle\int_{\partial\Sigma}\frac{4\sum_{i=1}^na_i^2x_i^2-2(1-c)}{|\nabla F|}\, ds \nonumber\\ && \\
&=& \int_{\partial\Sigma}\frac{4\left[1-c+\sum_{i=k+1}^n\left(-a_ix_i^2+a_i^2x_i^2\right)\right]-2(1-c)}{|\nabla F|}\, ds \nonumber	\\
                        &=& \int_{\partial\Sigma}\frac{\sum_{i=k+1}^n4a_i\left(a_i-1\right)x_i^2+2(1-c)}{|\nabla F|} \, ds\nonumber \\
                        &=& \int_{\partial\Sigma}\frac{\sum_{i=k+1}^n4a_i\left(a_i-1\right)x_i^2}{|\nabla F|}\, ds+(1-c)(n-1)|\Sigma|. \nonumber
\end{eqnarray}
where in the last equality we used Proposition \ref{Fhomo}. On other hand, the right side of (\ref{geral}) gives
\begin{eqnarray}
& &\int_\Sigma\left[2\sum_{i=1}^na_i(1-N_i^2)-(n-1)\right]\, d\Sigma\nonumber\\
&=& \int_\Sigma\left[2\left(\sum_{i=1}^k(1-N_i^2)+\sum_{i=k+1}^na_i(1-N_i^2)\right)-(n-1)\right]\, d\Sigma \nonumber\\
&=&\int_\Sigma\left[2\left(k-1+\sum_{i=k+1}^nN_i^2+\sum_{i=k+1}^na_i(1-N_i^2)\right)-(n-1)\right]\, d\Sigma \nonumber \\
&=&\int_\Sigma\left[2\left(n-1+\sum_{i=k+1}^n\left((N_i^2-1)+a_i(1-N_i^2)\right)\right)-(n-1)\right]\, d\Sigma \nonumber \\
&=& (n-1)|\Sigma|+2\int_\Sigma\sum_{i=k+1}^n(N_i^2-1)(1-a_i)\, d\Sigma. \nonumber 
\end{eqnarray}
Now, putting this information together in both sides of (\ref{geral}) gives
\begin{equation}\label{geralb=0}
  \int_{\partial\Sigma}\frac{\sum_{i=k+1}^n4a_i\left(a_i-1\right)x_i^2}{|\nabla F|}\, ds=c(n-1)|\Sigma|+2\int_\Sigma\sum_{i=k+1}^n(N_i^2-1)(1-a_i)\, d\Sigma.
\end{equation}
Then, once $c\leq0, a_i\in\left\{-1,0,1\right\}$ and $|N|=1$, we have 
\begin{equation}
 0\leq\int_{\partial\Sigma}\frac{\sum_{i=k+1}^n2a_i\left(a_i-1\right)x_i^2}{|\nabla F|}\, ds\leq\int_\Sigma\sum_{i=k+1}^n(N_i^2-1)(1-a_i)\, d\Sigma\leq0, \nonumber
\end{equation}
which implies that both terms inside the integrals are zero. In particular, 
$$(N_i^2-1)(1-a_i)=0, \ \ \ \forall i=k+1, k+2, ..., n.$$
The equality above is possible only if exactly one among all $a_i$'s is different from 1 and all others are equal to 1. Without loss of generality we can admit that $a_n\neq1$.

Finally, if we put $a_1=a_2=...=a_{n-1}=1$ and $a_n=-1$ in the inequality above we get $x_n=0$ and $N_n=1$ which implies that $\Sigma$ is the flat disk supported at the origin. On other hand, putting $a_1=a_2=...=a_{n-1}=1$ and $a_n=0$ gives  $N_n=1$ and then $\Sigma$ can be any of the flat disks intersecting $\partial\Omega$ orthogonally.
\end{proof}

Theorem \ref{all=1} assures that among all quadric domains, whose boundary is a level set of $F$ as in (\ref{b=0}), the only ones which contain free boundary minimal hypersurfaces are the hyperboloid of one sheet and the cylinder. Theorem \ref{all=1} also assures that does not exist free boundary minimal hypersurfaces on $\s^{n-k-1}\times\R^{k}$ (with  $k\geq2$) neither on slab (i.e. when $F$ is as in \ref{b=0} with all $a_i=0$ but $a_n=1$).

The next result shows that does not exist free boundary minimal hypersurfaces on cones neither on hyperboloid of two sheets.


\begin{theorem} \label{2sheets}
Suppose that  $F$ is given by 
\begin{equation}\label{c>1}
F(x_1,x_2,..., x_n)=\sum_{i=1}^{n-1}x_i^2-x_n^2+c.
\end{equation}
with $c\geq1$.Then,  does not exist minimal hypersurface with free boundary  on $\partial \Omega$.   
\end{theorem}

\begin{proof}

Suppose that do exist a minimal free boundary hypersurface $\Sigma$ on $\Omega$. We note that $\Sigma$ can not be on neither of the domains 
$$D_1= \left\{(x_1,..., x_n)\in\R^n; x_n\geq\sqrt{\sum_{i=1}^{n-1}x_i^2+c-1}\right\}$$
or
$$D_2=\left\{(x_1,..., x_n)\in\R^n; x_n\leq-\sqrt{\sum_{i=1}^{n-1}x_i^2+c-1}\right\}.$$  
Otherwise, we would have $\frac{\partial F}{\partial x_n}<0$ over $D_1$ or $\frac{\partial F}{\partial x_n}>0$ over $D_2$ and Corollary \ref{dF>0} applied to $\left.F\right|_{D_i}$ would give a contradiction. Then, $\Sigma$ must have boundary components on $\partial D_1$ and $\partial D_2$. 

Once $\Sigma$ is free boundary, we note that the conormal field of $\Sigma$, at any $x\in\partial\Sigma$, points out to the same direction as

$$\nabla F(x)=2\left(\sum_{i=1}^{n-1}x_iE_i-x_nE_n\right)=2\left(x-2x_nE_n\right),$$ 
where $\left\{E_1, E_2, ..., E_n\right\}$ is the canonical basis of $\R^n$.

Let $G:\R^n\rightarrow\R$ be the map
\begin{equation}\label{G}
G(x_1,x_2,..., x_n)=\sum_{i=1}^{n-1}x_i^2. \nonumber
\end{equation}
We note that $G^{-1}(l),\  l>0$, are regular sets which the unitary normal field over the level $l$ at $x\in\R^n$ is
\begin{equation}\label{NofG}
\mathcal{N}(x)=\frac{\nabla G}{|\nabla G|}=\frac{1}{\sqrt{l}}\sum_{i=1}^{n-1}x_iE_i=\frac{1}{\sqrt{l}}\left(x-x_nE_n\right), \nonumber
\end{equation}
Once  $\Sigma$ is compact, there is a $l_0>0$ such that $\Sigma\subset G^{-1}(l<l_0)$. We can decreasing $l_0>0$ until the surface $G^{-1}(l_0)$ touch $\Sigma$ for the first time at the point $\tilde{x}$. Then, $G^{-1}(l_0)$ and $\Sigma$ are tangent at $\tilde{x}$. In particular, $N(\tilde{x})=\mathcal{N}(\tilde{x})$. It follows that $\tilde{x}\notin\partial\Sigma$, otherwise we would have
$$0=\interno{N(\tilde{x})}{\nabla F(\tilde{x})}=\interno{\mathcal{N}(\tilde{x})}{\nabla F(\tilde{x})}=\frac{2}{\sqrt{l_0}}\left(|\tilde{x}|^2-\tilde{x}_n^2\right)>0.$$
Now, let $\rho=G|_{\Sigma}$. A straightforward calculation gives
$$\nabla\rho(x)=2x-2\interno{x}{E_n}E_n$$
and, for all $X,Y\in T\Sigma$,
{\small
\begin{eqnarray}
\text{Hess}_\Sigma\rho(X,Y)  &=& \text{Hess}_{\R^n}\rho(X,Y)+\interno{AX}{Y}\interno{\nabla \rho}{N} \nonumber	\\
                             &=& 2\left[\interno{X}{Y}-\interno{X}{E_n}\interno{Y}{E_n}+\interno{AX}{Y}\left(\interno{x}{N}-\interno{x}{E_n}\interno{N}{E_n}\right)\right]. \nonumber
\end{eqnarray} }
Once $\rho$ attains it's maximum at $\tilde{x}$, it follows that  
{\small\begin{eqnarray}
0\geq\frac{1}{2}\text{Hess}_\Sigma\rho(X,X)  &=& |X|^2-\interno{X}{E_n}^2+\interno{AX}{X}\left(\interno{\tilde{x}}{\mathcal{N}}-\interno{\tilde{x}}{E_n}\interno{\mathcal{N}}{E_n}\right) \nonumber	\\
                                  &=& |X|^2-\interno{X}{E_n}^2+\frac{1}{\sqrt{l_0}}\interno{AX}{X}\left(|\tilde{x}|^2-\tilde{x}_n^2\right).  \nonumber
\end{eqnarray} }
The inequality above then gives
\begin{equation}\label{A<0}
\interno{AX}{X}\leq-\frac{|X|^2-\interno{X}{E_n}^2}{|\tilde{x}|^2-\tilde{x}_n^2}\sqrt{l_0 }\leq 0, \  \ \ \forall\ X\in T_{\tilde{x}}\Sigma.
\end{equation}
But, once $\Sigma$ is a minimal hypersurface, the inequality (\ref{A<0}) is possible only if $\interno{AX}{X}=0, \ \ \forall X\in T_{\tilde{x}}\Sigma$. However, one can put $X\in T_{\tilde{x}}\Sigma$ with $X\neq 0$ and $\interno{X}{E_n}=0$ in (\ref{A<0}) (which indeed is possible because $T_{\tilde{x}}\Sigma=T_{\tilde{x}}G^{-1}(l_0)$)  and give the contradiction $\interno{AX}{X}<0$.   
\end{proof}

The last result of this section assures that there is no minimal hypersurfaces with free boundary neither on cylinders over cones or cylinders over hyperbola.  

\begin{theorem} \label{cilindcone}
Suppose that  $F$ is given by 
\begin{equation} 
F(x_1,x_2,..., x_n)=\sum_{i=1}^{n-2}x_i^2-x_n^2+c.
\end{equation}
with $c\geq1$.Then,  does not exist minimal hypersurface with free boundary  on $\partial \Omega$.   
\end{theorem}

\begin{proof}
We can follow the same steps like we have just done in Theorem \ref{2sheets} but using the functions $G(x)=x_{n-1}^2$.
\end{proof}
\vspace{1cm}



\section{Existence of Minimal Free-Boundary Hypersurfaces in rotational graphics}

Let us consider a rotational hypersurface in the following sense. Let $\alpha(t)=(f(t),t)$ be a plane curve $\alpha$ that is the graph of a positive real valued smooth function $f:I\longrightarrow\mathbb{R}$ in the $x_1x_{n+1}$-plane. Let $\Theta$ be a parametrization of the $(n-1)-$ dimensional unit sphere in the hyperplane $x_{n+1}=0$. The hypersurface of revolution $\Sigma$ with generatriz $\alpha$ can be parametrized by
$$X(\Theta,t)=(\Theta f(t),t).$$
In this scope, we discuss conditions to the existence of minimal free-boundary surfaces in domains $\Omega$ which boundary is a hypersurface of revolution. Let us denote $x=(x_1,\cdots,x_n)$ and $y=x_{n+1}$.  Let $F:\mathbb{R}^{n+1}=\mathbb{R}^n\times \mathbb{R} \to \mathbb{R}$ be the smooth function defined by

$$F(x,y)=|x|^2-f^{2}(y)+1,$$ 
we have that $\partial\Omega\subset F^{-1}(1)$. Notice  that $1$ is a regular value of $F$ because
$$\nabla F=(2x,-2ff').$$ Thus
$$|\nabla F|=2\sqrt{|x|^{2}+f^{2}(f')^{2}}.$$
In particular, if $\Sigma\subset\Omega$ is a hypersurface with $\partial\Sigma\subset\partial\Omega$, then
$$|\nabla F|\Big|_{\partial\Sigma}=2f\sqrt{1+(f')^{2}}.$$
We observe that

$$
D^{2}F=\left(\begin{array}{cc}
2I_n&0\\
0&-2(f')^{2}-2ff''
\end{array}\right)
$$
where $I_n$ means the identity map in $\mathbb{R}^n$.

In this setting, if $N=(N_{1},\cdots,N_{n+1})$ is the outward unit normal to $\Sigma$, and $\Sigma$ is minimal, we have that

\begin{eqnarray*}
\Delta_{\Sigma}F&=&\Delta_{\mathbb{R}^{n+1}}F-D^{2}F(N,N)\\
& = &2n-2(f')^{2}-2ff''-2(N_{1}^{2}+\cdots+N_{n}^{2}-[(f')^{2}+ff'']N_{n+1}^{2})\\
& = & 2n+2(N_{n+1}^{2}-1)-2[(f')^{2}+ff''](1-N_{n+1}^{2})
\end{eqnarray*}
where we have used that $N_{1}^{2}+\cdots+N_{n}^{2}=1-N_{n+1}^{2}$. Thus, using the Proposition\,\ref{intid}, we conclude that
\begin{eqnarray}  \label{eqgr1}
\int_{\partial\Sigma}f\sqrt{1+(f')^{2}}\,ds &=& \frac{1}{2} \int_{\partial\Sigma}|\nabla F|\,ds\ =\  \frac{1}{2} \int_{\Sigma} \Delta_{\Sigma}F \,d\Sigma  \nonumber\\
                                            &=& n|\Sigma|+\int_{\Sigma}(N_{n+1}^{2}-1)[(f')^{2}+ff''+1]\,d\Sigma. 
\end{eqnarray}
On another hand, using the identity (\ref{fbident}) in this case and observing that  the outward unit normal to $\partial\Omega$ is $\bar{N}=\frac{\nabla F}{|\nabla F|}$, we have that
\begin{equation}\label{eqgr2}
n|\Sigma|=\int_{\partial\Sigma}\langle (x,y),\bar{N}\rangle\,ds = \int_{\partial\Sigma}\frac{|x|^2-yff'}{f\sqrt{1+(f')^{2}}}\,ds =
\int_{\partial\Sigma}\frac{f-yf'}{\sqrt{1+(f')^{2}}}\,ds.
\end{equation}
Confronting (\ref{eqgr2}) and (\ref{eqgr1}), we have that
\begin{equation}
\int_{\partial\Sigma}\frac{f(f')^{2}+yf'}{\sqrt{1+(f')^{2}}}\,ds=\int_{\Sigma}(N_{n+1}^{2}-1)[(f')^{2}+ff''+1]\,d\Sigma.
\end{equation}
 
In particular, under some conditions in $f$, we will have nonexistence results of free-boundary minimal hypersurfaces.  

Note that if $\Sigma \subset \Omega'$ where $\Omega'$ is a piece of $\Omega$ with $f'\geq0$ and $f''\geq0$, then $\Sigma$ is a flat disk, and the boundary $\partial \Sigma$ is at $f'(t)=0$.
In general, we can conclude the following 

\begin{theorem}\label{rotationalgraphic}
Assume that $\Sigma \subset \Omega'$ where $\Omega'$ is a piece of $\Omega$ with $f'\geq0$. Then $\Sigma$ is a flat disk, and the boundary $\partial \Sigma$ is at $f'(t)=0$. 
\end{theorem}
As a direct consequence, we obtain the following uniqueness and non-existence result.
\begin{theorem}
The only bounded smooth immersed minimal hypersurface with free boundary on a catenoid is the flat disk supported at the origin. There does not exist any bounded smooth immersed minimal hypersurface with free boundary on a cone.There does not exist any immersed minimal hypersurface with free boundary on a hyperbolic paraboloid.
\end{theorem}

Note that there are no topological or symmetry restrictions on the minimal hypersurfaces on the result above.



\section{Free boundary surfaces in the ball}


Let $x:\Sigma\rightarrow\mathbb{B}^{3}$ be a free boundary embedding of a connected and oriented surface $\Sigma$ in the unitary ball $\mathbb{B}^{3}\subset\mathbb{R}^{3}$ centered at the origin. The orientability of $\Sigma$ let us choose a unitary normal field $N$ over all surface. As it is known, the Weingarten operator of $x$ is given by $A=-(\overline{\nabla}N)^T$, where $\overline{\nabla}$ is the connection of $\mathbb{R}^{3}$.

Putting $\text{dist}(\cdot, \cdot)$ as the intrinsic distance over $\Sigma$, we call a {\it $\varepsilon$-collar of $\partial\Sigma$} (or simply {\it collar}) the set $\Gamma_{\varepsilon}=\left\{p\in\Sigma; \text{dist}(p,\partial\Sigma)<\varepsilon\right\}$. 

Identifying $x$ with the position vector it is easy to see that the field $x^T$, which is the projection of $x$ over $T\Sigma$, is orthogonal to $\partial\Sigma$ because $\Sigma$ is free boundary. Besides, $x^T$ satisfies $|x^T|=|x|=1>0$ on each connected component $\gamma_i$ of $\partial\Sigma$. Therefore, the continuously of the field $x^T\in T\Sigma$ assures that exists a collar $\Gamma_i\subset\Sigma$ containing $\gamma_i$ whose
\begin{equation}
|x^T|   \geq \varepsilon_i>0 \ \ \ \text{on}\  \Gamma_i.
\label{eq:x>0}
\end{equation}
From now on we will assume that each collar in $\Sigma$ containing a connected component of $\partial\Sigma$ satisfies the condition (\ref{eq:x>0}).

In this section, two functions over $\Sigma$ are crucial:
$$g=\interno{x}{N}\ \ \ \ \text{and} \ \ \ \ \ \rho=\frac{1}{2}|x|^2.$$
As $\Sigma$ is free boundary on the sphere, $g\equiv0$ and $\rho\equiv1/2$ over the connected components of $\partial\Sigma$, that is, each connected component of $\partial\Sigma$ is level set for the functions  $g$ and $\rho$. In particular,
\begin{equation}\label{rhoandsigma}
\nabla g\perp\partial\Sigma \ \ \ \ \text{and}\ \ \ \ \ \nabla \rho\perp\partial\Sigma.
\end{equation}
In the sequel, for we establish the gradients of these functions, observe that, for any $X\in T\Sigma$,
$$
\begin{array}{rcl}
\interno{\nabla g}{X} &=& X \interno{x}{N} \\ 
                      &=&   \interno{\overline{\nabla}_Xx}{N} + \interno{x}{\overline{\nabla}_XN} \\
                      &=&   \interno{X}{N} + \interno{x}{-AX} \\ 
                      &=&   \interno{-Ax^T}{X}.
\end{array}
$$
Hence,  $\nabla g=-Ax^T$. On other hand,

$$
\begin{array}{rcl}
\interno{\nabla \rho}{X} &=& \frac{1}{2} X \interno{x}{x} \\ 
                         &=& \interno{\overline{\nabla}_Xx }{x} \\
                         &=& \interno{X}{x} \\
                         &=& \interno{x^T}{X}
\end{array}
$$
therefore, $\nabla \rho = x^T$.

In the next lemma we consider the operator defined for each $p\in\Sigma$ by
$$
\begin{array}{rll}
 J: T_p\R^3   &\longrightarrow& T_p\Sigma  \\ 
      V &\longmapsto&     J(V)=N\wedge V^T
\end{array}
$$
where $\wedge$ indicate the vectorial product operation in $\R^3$.

\begin{lemma}\label{lem1}
Let $x:\Sigma\rightarrow\R^3$  be a orientable surface and let $\omega=J(x)$. Then
$$\interno{\nabla_X\omega}{X}=\interno{x}{N}\interno{AX}{J(X)}, \ \ \forall\  X\in T\Sigma.$$
In particular, if $x^T$ is a principal direction of the Weingarten operator on a collar over a connected component of  $\partial\Sigma$, then
$$\interno{\nabla_{x^T}\omega}{x^T}=0.$$
\end{lemma}

\begin{proof}
Firstly we note that
$$
\begin{array}{rcl}
\overline{\nabla}_X\omega &=& \overline{\nabla}_X(N\wedge x) \\ 
                          &=& \overline{\nabla}_X N\wedge x+N\wedge\overline{\nabla}_Xx \\
                          &=& -AX\wedge x+N\wedge X \\   
                          &=& -AX\wedge(x^T+\interno{x}{N}N) -J(X)\\
                          &=& -AX\wedge x^T - \interno{x}{N} AX\wedge N -J(X)\\
                          &=& -AX\wedge x^T - \interno{x}{N} J(AX) -J(X). 
\end{array}
$$
Now, as $AX\wedge x^T$ is orthogonal to $\Sigma$ we have that

$$
\begin{array}{rcl}
\interno{\nabla_X\omega}{X} &=& \interno{\overline{\nabla}_X\omega}{X}\\
                            &=& -\interno{x}{N}\interno{J(AX)}{X}-\interno{J(X)}{X} \\ 
                            &=& \interno{x}{N}\interno{AX}{J(X)}.                          
\end{array}
$$
\end{proof}


In the next result, when it is said that $x^T$  is eigenvector (and consequently $\neq0$), we are assuming that the collar $\Gamma$ satisfies (\ref{eq:x>0}). There is no loss of generality in assuming this condition because whether it was not verified it would be true over the collar $\Gamma \cap \Gamma_i$. 

\begin{proposition}\label{pr1}
Let $\Sigma\subset\mathbb{B}^{3}$ be a embedded free boundary surface and let $\gamma\subset\partial\Sigma$ be a connected component of the boundary of $\Sigma$. Assume that exists
one collar $\Gamma\subset\Sigma$ over $\gamma$ whose the set  $\left\{x^T, Ax^T\right\}$ is linearly dependent. Then, the eigenvalue associated to $x^T$ is constant over $\gamma$.
\end{proposition}

\begin{proof}

Let $\lambda$ be the eigenvalue associated to $x^T$, i.e., $Ax^T=\lambda x^T$. We consider the functions $g$ e $\rho$. From (\ref{rhoandsigma}), we have that $\nabla g=\pm |\nabla g| x^T$ over $\partial\Sigma$. On other hand, we know that  $\nabla g=-Ax^T$ over $\partial\Sigma$. Then, 
$$\lambda=\interno{Ax^T}{x^T}=\interno{\pm|\nabla g|x^T}{x^T}=\pm |\nabla g|.$$
Therefore, in order to show that $\lambda$ is constant over $\partial\Sigma$ it is sufficient to check that
$$\omega|\nabla g|^2=0\ \ \ \ \ \text{over}\ \ \ \ \ \ \partial\Sigma,$$
where $\omega=J(x)=N\wedge x$ is the field in $T\Sigma$ which is tangent to $\partial\Sigma$. A straightforward calculation gives
$$
\begin{array}{rcl}
\omega|\nabla g|^2 &=& 2 \interno{\nabla_\omega\nabla g}{\nabla g} \\ 
                   &=& 2 \interno{\text{Hess}_g\omega}{\nabla g} \\
                   &=& 2 \interno{\omega}{\text{Hess}_g\nabla g} \\   
                   &=& 2 \interno{\omega}{\nabla_{\nabla g}\nabla g}\\
                   &=& 2 \nabla g\interno{\omega}{\nabla g}-2\interno{\nabla_{\nabla g}\omega}{\nabla g} \\
                   &=& 2 \nabla g\interno{J(x^T)}{-Ax^T}-2\interno{\nabla_{-Ax^T}\omega}{-Ax^T} \\
                   &=& 2 \nabla g\interno{J(x^T)}{-Ax^T}-2\lambda^2\interno{\nabla_{x^T}\omega}{x^T}.
\end{array}
$$
We note that both terms of the last equality are zero, the first one because the vectors   $J(x^T)$ and $Ax^T$ are orthogonal over each $\Gamma_i$, and the socond one because of the Lemma \ref{lem1}
\end{proof}



In the discussion that will came next,  we will need get informations from the  Weingarten operator of $\partial\Sigma$ seen as a submanifold of  $\Sigma$ and seen as a submanifold of $\mathbb{S}^2$. The dimension of  $\Sigma\subset\mathbb{B}^{n+1}\subset\R^{n+1}$ does not make any difference here. Then, we will assume for a while that $n\in\mathbb{N}$ is arbitrary.

For we be able to distinguish each Weingarten operator we will use the following notation: $A_{\partial \Sigma}^{\mathbb{S}^{n}}$ (for $\partial \Sigma^{n-1}\hookrightarrow\mathbb{S}^{n}$); $A_{\mathbb{S}^{n}}^{\R^{n+1}}$   (for $\mathbb{S}^{n}\hookrightarrow\mathbb{R}^{n+1}$) and $A_{\partial \Sigma}^{\Sigma}$ (for $\partial\Sigma^{n-1}\hookrightarrow\Sigma^{n}$).

Now, for each $X,Y\in T\partial\Sigma$,

\begin{eqnarray}
\nabla_X^{\R^{n+1}}Y &=& \nabla_X^{\s^{n}}Y+\interno{A_{\s^{n}}^{\R^{n+1}}X}{Y}x \label{ball1}\\ 
                     &=& \nabla_X^{\partial \Sigma}Y+\interno{A_{\partial \Sigma}^{\s^{n}}X}{Y}N-\interno{X}{Y}x. \nonumber
\end{eqnarray}

\begin{eqnarray}
\nabla_X^{\R^{n+1}}Y &=& \nabla_XY+\interno{AX}{Y}N \label{ball2}\\ 
                     &=& \nabla_X^{\partial \Sigma}Y+\interno{A_{\partial \Sigma}^{\Sigma}X}{Y}x+\interno{AX}{Y}N.\nonumber 
\end{eqnarray}

Therefore, $A=A_{\partial\Sigma}^{\s^n} \ \ \ \text{over} \ \ \partial\Sigma$. For each $p\in\partial \Sigma$ let $\{e_1, .., e_{n-1}\}\subset T_p\partial \Sigma$ be an  eigenvector orthonormal basis of $A_{\partial \Sigma}^{\mathbb{S}^{n}}$ associated to eigenvalues $\tau_1, ..., \tau_{n-1}$, respectively. The matrix of $A$ in the basis $\{e_1, .., e_{n-1}, x^T\}\subset T_p \Sigma$ is

\begin{equation}\label{matn} 
A=\left(
\begin{array}{ccccc}
\tau_1      &   0     & \cdots  &    0          & 0\\
   0        & \tau_2  & \cdots &    0          & 0\\
\vdots      & \vdots  & \ddots  & \vdots       &\vdots        \\
   0        &   0     & \cdots  & \tau_{n-1}   & 0\\   
   0        & 0       &\cdots   & 0            & \lambda 
\end{array}
\right)
\end{equation}

where $\lambda$ is the eigenvalue associated to the eigenvector $x^T$ (we still assuming $\Sigma\subset\mathbb{B}^{n+1}$ free boundary). In particular when $n=2$, the matrix of $A$ became
$$
A=\left(
\begin{array}{cc}
\tau    &  0 \\
   0    & \lambda 
\end{array}
\right),
\label{mat2} 
$$ 
where $\tau$ is the geodesic curvature of $\partial\Sigma$ seen as submanifold of  $\s^2$. Therefore,  the mean curvature of $\Sigma\subset\mathbb{B}^2$ is

\begin{equation}\label{media}
H=\frac{\tau+\lambda}{2}.
\end{equation}

\begin{theorem}[Theorem B]\label{pr2}
Under the same hypothesis of Proposition \ref{pr1}, if  $\Sigma$ is minimal then $\Sigma$ is either the equatorial disk $\mathbb{D}$ or the critical catenoid $\mathbb{K}$.
\end{theorem}

\begin{proof}


We have shown in Proposition \ref{pr1} that $\lambda$ is constant over $\gamma$. Hence, as $\Sigma\subset\mathbb{B}^3$ is minimal, it follows from (\ref{media}) that the geodesic curvature $\tau$ of $\gamma\subset\s^2$ is also constant. Then, $\gamma$ is a circle over $\s^2$ and, thus, rotationally invariant. Therefore, the conclusion follows the same idea of Corollary 3.9 in \cite{kapouleas2017}.

\begin{remark}
A similar result can be obtained by considering general $n$. In fact, since $\Sigma$ has constant mean curvature and $\lambda$ is constant, we obtain from \eqref{matn} that the mean curvature of hypersurface $\Gamma\subset\partial\Sigma$ is also constant. Then, by classical Alexandrov's theorem (see, for example \cite{ros1987}), $\Gamma$ is a hypersphere. This produces the following result:
\vspace{0.5cm}
\begin{quote}
Let $\Sigma^n$ be a free-boundary, constant mean curvature hypersurface in $\mathbb{B}^{n+1}$. Suppose that the associated eigenvalue to $x^T$ by $A$ is constant along of a connected component $\Gamma\subset\partial\Sigma$. Then $\Gamma$ is a hypersphere of dimension $n-1$.  
\end{quote}
\vspace{0.5cm}
In order to prove the rigid version, as in Theorem \ref{pr2}, in this underlying context, it would be necessary analysing, for example, a higher dimensional version of \cite{kapouleas2017} and \cite{bjorling}.
\end{remark}


\end{proof}



\section{Gap results for free boundary surfaces in Ellipsoids}

The aim of this section is characterize free boundary minimal hypersurfaces $\Sigma$ on rotational ellipsoids (Theorem \ref{prop2}). For we being able to do this we will need some tools, mostly of then concerning the function $F$ below.           

Let $F:\R^3\rightarrow\R$ be the function
$$F(x_1,x_2,x_3)=\frac{x_1^2}{a^2}+\frac{x_2^2}{a^2}+\frac{x_3^2}{b^2}$$
with $a\geq b$. The level set $F^{-1}(1)$ is an ellipsoid of revolution with axes $2a$ and $2b$, whose revolution has been made over the $x_3$ axis.

A straightforward calculation gives the following expression to the gradient field of $F$ with respect to $\R^3$:
$$\overline{\nabla}F(x)=\frac{2}{a^2}\left[x+\left(\frac{a^2}{b^2}-1\right)\interno{x}{E_3}E_3\right]$$
where $E_3=(0,0,1)\in\R^3$. Now, in order to establish the Hessian of $F$ in $\Sigma$ we need the auxiliary functions on $\R^3$: $\phi=\frac{|x|^2}{2}$ and $\phi_3=\interno{x}{E_3}$. For all $X\in\Sigma$ we have
$$\text{Hess}_{\R^3}F(X,Y)=2\interno{X}{Y}+2\left(\frac{a^2}{b^2}-1\right)\interno{X}{E_3}\interno{Y}{E_3}$$
$$\text{Hess}_{\Sigma}F(X,Y)=\text{Hess}_{\R^3}F(X,Y)+\interno{AX}{Y}\interno{\overline{\nabla} F}{N}$$
$$\interno{\nabla\phi}{X}=X\phi=\interno{x}{X}=\interno{x^T}{X},$$
where $x^T\in T\Sigma$ is the field obtained from the projection of $x$ over $T_x\Sigma$. Therefore $\nabla\phi=x^T$ and using this we get   

\begin{eqnarray}
\text{Hess}_{\Sigma}\phi(X,Y) &=& \interno{\nabla_X\nabla\phi}{Y} \nonumber\\
                              &=& \interno{\nabla_Xx^T}{Y} \nonumber\\
                              &=& \interno{\overline{\nabla}_X(x-\interno{x}{N}N)}{Y}  \nonumber\\
                              &=& \interno{X}{Y}+\interno{AX}{Y}\interno{x}{N} \nonumber
\end{eqnarray}
for all $X,Y\in T\Sigma$. Similarly we have \ $\nabla\phi_3=E_3^T$\  and \ {\small $\text{Hess}_{\Sigma}\phi_3(X,Y)=\interno{AX}{Y}\interno{N}{E_3}$}. Hence, using this expressions, we get

\begin{eqnarray} 
\text{Hess}_{\Sigma}F(X,Y) &=& \interno{\nabla_X\nabla F}{Y} \nonumber \\
                           &=& \interno{\nabla_X(\overline{\nabla} F)^T}{Y} \nonumber \\
                           &=& \frac{2}{a^2} \interno{\nabla_X\left(x^T+\left(\frac{a^2}{b^2}-1\right)\interno{x}{E_3}E_3^T\right)}{Y}  \nonumber \\                    
                           &=& \frac{2}{a^2} \left[\text{Hess}_{\Sigma}\phi(X,Y)+\left(\frac{a^2}{b^2}-1\right)\interno{X}{E_3}\interno{Y}{E_3}\right. \nonumber \\                                                  && \left.+ \left(\frac{a^2}{b^2}-1\right)\interno{x}{E_3}\text{Hess}_\Sigma\phi_3(X,Y) \right] \nonumber \\                           
                           &=& \frac{2}{a^2} \left[\interno{X}{Y}+\interno{AX}{Y}\interno{x}{N} + \left(\frac{a^2}{b^2}-1\right)\interno{X}{E_3}\interno{Y}{E_3}\right. \nonumber \\                                 && \left. + \left(\frac{a^2}{b^2}-1\right)\interno{x}{E_3}\interno{AX}{Y}\interno{N}{E_3} \right] \nonumber \\                           
                           &=& \frac{2}{a^2} \left[ \interno{X}{Y}+\interno{AX}{Y}\left(\interno{x}{N}+\left(\frac{a^2}{b^2}-1\right)\interno{x}{E_3}\interno{N}{E_3}\right)\right. \nonumber \\                 && \left. + \left(\frac{a^2}{b^2}-1\right)\interno{X}{E_3}\interno{Y}{E_3}     \right] \nonumber \\     
                           &=& \frac{2}{a^2} \left[\interno{X}{Y}+\interno{AX}{Y}g(x)+\left(\frac{a^2}{b^2}-1\right)\interno{TX}{Y}\right] \nonumber  \\     
                           && \nonumber        
\end{eqnarray}
where $g:\Sigma\rightarrow\R$ and $T:T_x\Sigma\rightarrow T_x\Sigma$ are, respectively,
\begin{eqnarray} 
g(x) &=& \interno{x}{N}+\left(\frac{a^2}{b^2}-1\right)\interno{x}{E_3}\interno{N}{E_3}; \ \text{and} \nonumber \\
TX   &=&  \interno{X}{E_3^T}E_3^T \nonumber 
\end{eqnarray}
It is easy to check that $T$ is a self adjoint operator whose $E_3^T$ is an eigenvector associated with the eigenvalue  $|E_3^T|^2$. Besides, we can take any nonzero vector in $T\Sigma$, orthogonal to $E_3^T$, to verify that zero is also an eigenvalue of $T$. Therefore
\begin{equation}\label{T>0}
0\leq \interno{TX}{X}\leq |E_3^T|^2|X|^2, \ \ \forall X\in T_x\Sigma.
\end{equation}

The next two lemmas will be useful in the proof of Proposition \ref{prop2}.


\begin{lemma}\label{lemma1}
For each $x\in\Sigma$, the eigenvalues of $\frac{a^2}{2}$Hess$_\Sigma F (x)$ are greater of equal to
$$1-\frac{|A|(x)}{\sqrt 2}g(x)\ \ \ \text{and} \ \ \ 1+\frac{|A|(x)}{\sqrt 2}g(x).$$
\end{lemma}
\begin{proof}
It follows from (\ref{T>0}) that
\begin{eqnarray}
\frac{a^2}{2}\text{Hess}_\Sigma F(X,Y)  &=&  \interno{X}{Y}+\interno{AX}{Y}g(x)+\left(\frac{a^2}{b^2}-1\right)\interno{TX}{Y}  \nonumber \\
                                      &\geq &  \interno{X+g(x)AX}{Y}. \nonumber
\end{eqnarray}
But, the eigenvalues of \ \ $X\mapsto X+g(x)AX$\ \ are
$$1+\lambda_1g(x)\ \ \ \text{and}\ \ \ 1+\lambda_2g(x),$$
where $\lambda_1\leq\lambda_2$ are the eigenvalues of $A$. The lemma now follows from the fact that $\Sigma$ minimal imply
$$\lambda_1=-\frac{|A|(x)}{\sqrt2}\ \ \ \text{and}\ \ \ \lambda_2=\frac{|A|(x)}{\sqrt2}.$$
\end{proof}


\begin{lemma}\label{kappa>0}
The Weingarten operator $A_{\partial\Omega}^{\R^3}$ of the ellipsoid $F^{-1}(1)=\partial\Omega$ in $\R^3$ satisfies 
$$\interno{A_{\partial\Omega}^{\R^3}X}{X}\geq c |X|^2>0, \ \ \forall X\in T\partial\Omega,\ X\neq0,$$
where $c>0$.
\end{lemma}

\begin{proof}
We claim that both eigenvalues $\kappa_1\leq\kappa_2$ of $A_{\partial\Omega}^{\R^3}$ are positive. Let $U\subset\R^2$ be an open set and $x:U\subset\R^2\rightarrow V\subset\partial\Omega$ the immersion
$$x(\theta,t)=(a\cos t\cos \theta, a\cos t\sin \theta, b\sin t),\ \ (\theta, t)\in U.$$
A straightforward calculation shows that the Gaussian curvature of $\partial\Omega$ at $x(\theta, t)$ is
$$K(\theta,t)=\frac{b^2}{(a^2\sin^2t+b^2\cos^2t)^2}.$$  
Hence, $K$ is strictly positive on $\partial\Omega$. In particular, $\kappa_1$ and $\kappa_2$ have the same sign. Let $H$ be the mean curvature of $\partial\Omega\subset\R^3$. It is well known that
$$\kappa_1=H-\sqrt{H^2-K}\ \ \ \ \text{and}\ \ \ \ \kappa_2=H+\sqrt{H^2-K}.$$
We claim that $H>0$. Indeed, if $H<0$ we will have $\kappa_1<0$ and
$$\kappa_2=H+\sqrt{H^2-K}>H+|H|=0$$
which gives a contradiction. Hence $H>0$ and, once $\kappa_2=H+\sqrt{H^2-K}>0$, we have $\kappa_1>0$ either. 

Now, once $\kappa_1$ is a continuous function on the compact set $\partial\Omega$, we have that $c=\min_{x\in\partial\Omega}\kappa_1(x)>0$. Therefore, for all  $X\in T\partial\Omega$ with $X\neq0$,
$$\interno{A_{\partial\Omega}^{\R^3}X}{X}\geq \kappa_1|X|^2\geq c |X|^2>0.$$
\end{proof}


\begin{theorem}[Theorem C]\label{prop2} \label{elipsoide}
Let $\Sigma^2$ be a compact minimal surface with free boundary in the ellipsoid $F^{-1}(1)$. If $|A|^2g(x)^2\leq2$ on $\Sigma$, then $\Sigma$ is homeomorphic to a disk or an annulus.
\end{theorem}

\begin{proof}


The proof follows the same steps of \cite{ambrozio2016}. First, we define
$$\mathcal{C}=\left\{p\in\Sigma;\  F(p)=\min_{x\in\Sigma}F(x)\right\}.$$
Given $p, q\in\mathcal{C}$, let $\gamma:[0,1]\rightarrow\Sigma$ be a geodesic such that $\gamma(0)=p$ and $\gamma(1)=q$. It follows from Lemma \ref{lemma1} that  the condition $|A|^2g(x)^2\leq2$ is equivalent to Hess$_\Sigma F\geq0$ on $\Sigma$. Then, $(F\circ \gamma)''(t)\geq0$ for all $t\in[0,1]$. Hence, $F\circ \gamma$, whose minima are attain at $t=0$ and $t=1$, is convex on $[0,1]$. We conclude that $(F\circ\gamma)(t)\equiv\min_\Sigma F$. Therefore, $\gamma([0,1])\subset\mathcal{C}$ and $\mathcal{C}$ must be a totally convex subset of $\Sigma$. In particular, totally convex property of $C$ also assures that $\gamma([0,1])\subset C$ for all geodesic loop $\gamma:[0,1]\rightarrow\Sigma$, based at a point $p\in C$.

Now we claim that the geodesic curvature $k_g$ of $\partial\Sigma$ in $\Sigma$ is positive. In fact,  if we do a similar decomposition as in equations (\ref{ball1}) and (\ref{ball2})  we will have  $A_{\partial\Sigma}^\Sigma=A_{\partial\Omega}^{\R^3}$ on $\partial\Sigma$, where $\partial\Omega$ is the ellipsoid $F^{-1}(1)$. Hence, if $X\in T\partial \Sigma$ is unitary, it follows from Lemma \ref{kappa>0} that 
\begin{eqnarray}
k_g &=& \interno{\overline{\nabla}_XX}{\nu} = X\interno{X}{\nu}-\interno{X}{\overline{\nabla}_X\nu} = \interno{A_{\partial\Sigma}^\Sigma X}{X} \nonumber \\ 
    &=&  \interno{A_{\partial\Omega}^{\R^3} X}{X}\geq c>0. \nonumber
\end{eqnarray}

This assures that each geodesic $\gamma$ which connect two points in $\mathcal{C}$ is completely inside of $\Sigma$, that is, the trace of $\gamma$ does not have points of $\partial\Sigma$. Hence, $\mathcal{C}$ is contained in the interior of $\Sigma$. 

If $C$ consists of a singular point $p\in\Sigma\backslash \partial\Sigma $ and there is a non-trivial homotopy class $[\alpha]\in\pi_1(\Sigma,p)$, we can find a geodesic loop $\gamma:[0,1]\rightarrow\Sigma, \gamma(0)=\gamma(1)=p$ with $\gamma\in[\alpha]$. But, since $C$ is totally geodesic, $\gamma([0,1])\subset C$ and, in particular, $C$ has more than one point, which is a contradiction. Therefore, if $C$ contains only a single point, $\Sigma$ is homeomorphic to a disk.

If $C$ has more than one point and it is not homeomorphic to a disk, we can find a  geodesic loop $\gamma:[0,1]\rightarrow\Sigma, \gamma(0)=\gamma(1)$, belonging to a non-trivial homotopy class $[\alpha]\in\pi_1(\Sigma,p)$. In this case we must have $\gamma'(0)=\gamma'(1)$. Indeed, if $\gamma'(0)\neq\gamma'(1)$ the totally  convexity of $C$ together with $\gamma([0,1])\subset C$ assure that it is possible to join points on $\gamma$ near the break at $p$ by minimising geodesics and find an open set $U\subset C$ which is a contradiction with the hypothesis of $\Sigma$ being minimal. By similar arguments we can also show that $\gamma([0,1])$ is a simple curve and that $C=\gamma([0,1])$. We conclude that $\pi_1(\Sigma,p)=\mathbb{Z}$. Then, $\Sigma$ is  homeomorphic to an annulus.


\end{proof}

\begin{remark}
After the conclusion of this work, we take note of \cite{seo}, where the authors address topological gap results where the ambient space is a strictly convex domain in a 3-dimensional Riemannian manifold with sectional curvature bounded above by a constant.
\end{remark}

\begin{remark}
We observe that similar techniques can be applied to approach more general results of free boundary constant mean curvature submanifolds (of arbritary codimension) immersed in higher dimensional ellipsoids.
\end{remark}


By the end, we also present a gap result in the unit ball.

\begin{theorem}[Theorem D] \label{cherndocarmokobayashi}Let $\Sigma^n$ be a free boundary minimal hypersurface in  the  unit ball $\mathbb{B}^{n+1}$. If $|A|^2\leq 2n$,  then $\Sigma$ is a totally geodesic equatorial disk $\mathbb{D}^n$.
\end{theorem}

\begin{proof}
Firstly, we observe that to recover the case 
of $(n+1)$-dimensional unit ball domain, we can get the function $F(x)=|x|^{2}$ in the formula \ref{fund} to obtain

\begin{equation}\label{firstconclusion}
\int_{\partial \Sigma} \varphi  =n \int_\Sigma \varphi +\frac{1}{2} \int_\Sigma (1-|x|^2) \Delta\varphi,
\end{equation}

It is well known that the support function $g=\langle x, N \rangle$
satisfies the following equation 
\[ \Delta g + |A|^2 g =0,
\]
Defining $J:=-\Delta - |A|^2$, called the Jacobi operator, $g$ is an eigenfunction associated to the eigenvalue $0$. If $g$ is a sign-changing function, then $0$ is not the first eigenvalue. So $\lambda_1^J<0$. Let us denote by $v$ the eigenfunction associated to $\lambda_1^J$. Plugging $v$ into \eqref{firstconclusion}, we have
\begin{equation}
\begin{array}{rcl}
\int_{\partial \Sigma} v & = &n \int_\Sigma v + \frac{1}{2}\int_\Sigma (1-|x|^2) \Delta v \\ \\
0 &=& n \int_\Sigma v + \frac{1}{2} \int_\Sigma (1-|x|^2)  \left(-|A|^2 v- \lambda_1^J v \right)\\ \\
&=& \int_\Sigma \left(2n-|A|^2\right)v + \int_\Sigma |x|^2 |A|^2 v- \int_\Sigma (1-|x|^2)\lambda_1^J v 
\end{array}
\end{equation}
This is impossible if $|A|^2\leq 2n$. We conclude that $g$ is nonnegative.
Now plugging $g$ into \eqref{firstconclusion},

\begin{equation}
\begin{array}{rcl}
\int_{\partial \Sigma} g & = &n \int_\Sigma g+ \frac{1}{2}\int_\Sigma (1-|x|^2) \Delta g \\ \\
&=& \int_\Sigma \left(n-\frac{1}{2}|A|^2\right)g + \frac{1}{2}\int_\Sigma |x|^2 |A|^2 g.
\end{array}
\end{equation}
Since $\Sigma$ is free boundary, we conclude that $g\equiv 0$ on the boundary, therefore if $|A|^2 \leq 2n$ then $g=0$ a.e., the continuity of $g$ on the other hand implies $g\equiv 0 $ in $\Sigma$. Therefore, $\Sigma$ is a minimal cone, hence by regularity  $\Sigma$ is a totally geodesic hypersurface.
\end{proof}

\begin{remark}
We mention that the gap of previous theorems is actually stronger, for $n< 5$, than an earlier result of the
first author with his collaborators  (see Theorem 1.1 in \cite{rosivaldo}, where the gap is $|A|^2\leq \frac{n^2}{2}$).
\end{remark}

\section*{Acknowledgments}
This work was partially funded by Conselho Nacional de Desenvolvimento Científico e Tecnológico (CNPq) Grants 312598/2018-1 (E. Barbosa), 316080/2021-7 (A. Freitas) and 312866/2018-6 (F. Vitório). This work also was funded by Public Call 03 Produtividade PROPESQ/PRPG/UFPB proposal code PIA13495-2020 and Paraíba State Research Foundation-Programa Primeiros Projetos Grant 2021/3175 (A. Freitas). This work is a part of the Ph.D. thesis of the third author.

\bibliographystyle{amsplain}

\end{document}